\documentclass[ECP]{ejpecp} 

\SHORTTITLE{A Maximal Inequality for Supermartingales} 

\TITLE{A Maximal Inequality for Supermartingales
\thanks{Supported by the National Science Foundation under Grant ECCS 10-28464.}
  } 

\AUTHORS{%
 Bruce Hajek\footnote{University of Illinois at Urbana-Champaign, USA.
    \EMAIL{b-hajek@illinois.edu}}
}

\KEYWORDS{maximal inequality; semimartingale calculus; drift; martingale} 

\AMSSUBJ{60G40;93E20;60G42;60G44} 

\SUBMITTED{January 2, 2014} 
\ACCEPTED{July 30, 2014}  

\VOLUME{0}
\YEAR{2014}
\PAPERNUM{0}
\DOI{vVOL-PID}

\ABSTRACT{A tight upper bound is given on the distribution of the maximum of a  supermartingale.
Specifically, it is shown that if $Y$ is a semimartingale with initial value zero and
quadratic variation process $[Y,Y]$ such that $Y + [Y,Y]$ is a supermartingale, then the probability the maximum
of $Y$ is greater than or equal to a positive constant $a$ is less than or equal to
$1/(1+a).$   The proof makes use of the semimartingale calculus and is inspired by dynamic programming.}

\numberwithin{equation}{section}
\theoremstyle{plain}

\newtheorem{cond}{Condition}[section]   

\newcommand{\IZ}{\mathbb{Z}}
\newcommand{\IR}{\mathbb{R}}

\begin{document}

\section{Introduction}

The basic framework of the classical martingale calculus will be used, as described, for example, in \cite{Jacod79,Kallenberg02,Meyer76,Protter04,wonghajek85}.
All random processes are assumed to be defined on a complete probability space $(\Omega, {\cal F}, P)$ with a
filtration of $\sigma$-algebras  $( {\cal F}_t : t \geq  0)$ that is assumed to satisfy the usual conditions of  right-continuity and inclusion of all sets of
probability zero.     Semimartingales are processes that are  c\`{a}dl\`{a}g (right continuous with  finite left limits)
and can be represented as the sum of a  c\`{a}dl\`{a}g
 local martingale and a  c\`{a}dl\`{a}g, adapted process of locally finite variation.
The quadratic variation process of a semimartingale  $Y$, denoted by $[Y,Y],$  is the nondecreasing process defined by
$$
[Y,Y]_t = \lim_{n\rightarrow\infty} \sum_{i=0}^{k_n-1}  (Y_{t^n_{i+1}}-Y_{t^n_i})^2
$$
for any choice of $0=t_0^n < t_1^n< \cdots < t_{k_n}^n = t$ such that
$\max_i |t^n_{i+1}-t^n_i|\rightarrow 0$ as $n\rightarrow\infty.$   The process $[Y,Y]$ can be decomposed
as $[Y,Y]_t=\sum_{s\leq t} (\triangle Y_s)^2  + [Y,Y]^c_t,$ where $\triangle Y_s$ is the jump of $Y$ at $s$, and
$[Y,Y]^c$ is the continuous component of $[Y,Y].$

Recall that a process $Y$ is a supermartingale if $E[Y_t-Y_s| {\cal F}_s ] \leq 0$ for $0\leq s \leq t.$   It
means that the process has a downward drift.   The following condition is stronger than the supermartingale condition,
requiring that the downward drift be at least as strong as a constant $\gamma$ times the rate of variation of the process,
as measured by the quadratic variation process. 

\begin{cond}  \label{cond.drift_down}
The process $Y$ is a semimartingale with $Y_0=0$, and $\gamma\geq 0$ such that $Y+\gamma [Y,Y]$ is a supermartingale.
\end{cond}

Condition \ref{cond.drift_down} is a joint condition on $Y$ and $[Y,Y]$, and is preserved under
continuous, adapted time changes.  A stronger condition, involving separate constraints on $Y$ and $[Y,Y]$, is the
following.
\begin{cond}  \label{cond.timed_drift_down}
$Y$ is a random process with $Y_0=0$, and $\gamma >0$,  $\mu>  0$ and $\sigma^2\geq 0$  are such that $\gamma=\frac{\mu}{\sigma^2}$
and  $(Y_t +\mu t : t \geq 0) $ and $([Y,Y] - \sigma^2 t : t\geq 0)$ are supermartingales.
\end{cond}
Condition \ref{cond.timed_drift_down} implies Condition \ref{cond.drift_down}
because, under Condition \ref{cond.timed_drift_down}, \\
$Y_t  + \gamma [Y,Y]_t=(Y_t+\mu t) + \gamma([Y,Y]_t-\sigma^2 t).$
Let $Y^*=\sup\{ Y_t : t\geq 0\}.$

\begin{proposition} \label{prop.drift_down}
Suppose $Y$ and $\gamma > 0$ satisfy Condition \ref{cond.drift_down} (or Condition \ref{cond.timed_drift_down})
and $a\geq 0.$  \\
(a)  The following holds:
\begin{equation}  \label{eq.mg_bound_down}
P\{  Y^*  \geq a \}   \leq \frac{1}{1+\gamma a}.
\end{equation}
(b)  Equality holds in \eqref{eq.mg_bound_down} if and only if the following is true, with
$T=\inf \{t \geq 0: Y_t \geq a\}:$
$(Y_{t\wedge T}: t\geq 0)$ has no continuous martingale component,  $Y$ is sample-continuous over
$[0,T)$ with probability one, $P(Y_T=a | T < \infty)=1$, and $(Y + \gamma [Y,Y])_{t\wedge T}$ is a martingale.
\end{proposition}

Proposition \ref{prop.drift_down} raises the question of whether there is a finite upper bound on $E[Y^*]$ depending only
on $\gamma.$   Also, the necessary conditions in Proposition \ref{prop.drift_down}(b)
cannot be satisfied for two distinct strictly positive values of $a.$    This raises the question as to how close to equality
the bound \eqref{eq.mg_bound_down} can be for all values of $a$, for a single choice of $Y$ not depending on $a$.
The following proposition addresses these two questions.

\begin{proposition}  \label{prop.uniform}
Given $\gamma \geq 0$ there exists $Y$ satisfying Condition \ref{cond.timed_drift_down} such that
\begin{equation} \label{eq.5bound}
P\{Y^* \geq a \}  \geq \frac{1}{5(1+a\gamma)}
\end{equation}
for all $a\geq 0.$   In particular, $E[Y^*]=+\infty$ for this choice of $Y.$
\end{proposition}

The remainder of this paper is organized as follows.   Section \ref{sec.construction} describes a
construction, showing there is a process meeting the conditions of Proposition \ref{prop.drift_down}(b),
and providing a proof of Proposition \ref{prop.uniform}.   Proposition  \ref{prop.drift_down}
is reformulated and proved in Section \ref{sec.dynamic_programming}.   A discrete time version of
Proposition \ref{prop.drift_down} is stated and proved in Section \ref{sec.discrete_time}.  Discussion
about intuition and possible extensions is given in Section \ref{sec.discussion}.  Some well known
related inequalities are discussed in Section \ref{sec.related}.

\section{The big jump construction}   \label{sec.construction}

A construction is given for a process meeting the bound of Proposition \ref{prop.drift_down} with
equality,   and for a  process providing  a proof  of Proposition \ref{prop.uniform}.  These processes
satisfy Condition \ref{cond.timed_drift_down}.

Let $\mu >0$, $\sigma^2 >0,$ and let $h: \IR_+ \rightarrow \IR_+$ be such that $h$ is nondecreasing.  
Consider a  random process $Y$ of the form:
$$
Y_t = \left\{   \begin{array}{cl}   -y(t)  &  t < T   \\
                                                  h(y(T)) - \mu (t -T)  &  t \geq T 
                                                  \end{array} \right.
$$
for a deterministic, continuous function $y=(y(t) : t\geq 0)$ and random variable $T$ described below.
Note that if $T <+\infty$ then $Y_{T-}=-y(T)$, $Y_T=h(y(T)),$ and $\triangle Y_T=y(T)+ h(y(T)).$  
Let $y$ be a  solution to the differential equation
$$
\dot{y}= \mu + \frac  {\sigma^2}{y+h(y)},  ~~~~ y(0)=0
$$
 and let $T$ be an extended nonnegative random variable such that for all $t\geq 0,$
$$
P\{ T\geq t  \} = \exp\left( -  \int_0^t \kappa (y_s) ds \right) ~~~\mbox{with}~~\kappa(y)  = \frac{\sigma^2}{(y+h(y))^2}.
$$
The function $\kappa(y(t))$ is the failure rate function of $T$: $P( T\leq t+\eta | T\geq t) = \kappa(y(t))\eta + o(\eta).$
The function $\kappa$ was chosen so that
\begin{eqnarray*}
E[ (Y_{t+\eta} - Y_t)^2  | T>t ]    & = &  (y(t)+h(y(t) ))^2\kappa (y(t))\eta + o(\eta)   \\
& = &   \sigma^2\eta + o(\eta) 
\end{eqnarray*}
and the differential equation for $y$ was chosen so that
\begin{eqnarray*}
E[Y_{t+\eta} - Y_t | T>t ]   & = &   - \dot{y}(t)\eta   +   (y(t)+h(y(t))\kappa(y(t) ) \eta   + o(\eta)  \\
& = &  -\mu \eta +  o(\eta)
\end{eqnarray*}
Therefore, $Y$ satisfies Condition \ref{cond.timed_drift_down}.

If the function $h$ is strictly increasing, then for any $c \geq 0$, a change of variable of integration
from $t$ to $y$ yields:
\begin{eqnarray}
P\{ Y^* \geq  h(c)  \}  & = &   P\{ T  \geq y^{-1}(c) \} - P\{ T= \infty \}  \nonumber  \\
& = &  \exp( -I(c) )  - \exp( -I(\infty ) )   \label{eq.Idef}
\end{eqnarray}
where
\begin{eqnarray*}
I(c) & = & \int_0^{y^{-1}(c)}  \kappa ( y(t)) dt     \\
& = &  \int _0^c    \kappa(y) \left\{ \mu + \frac{\sigma^2}{y+h(y)} \right\}^{-1} dy  
\end{eqnarray*}

\paragraph{Example 1: Meeting Proposition \ref{prop.drift_down} with equality.}
Take  $h(y)\equiv a$ for some $a > 0.$    We don't use \eqref{eq.Idef} because
$h$ is not strictly increasing, but similar reasoning yields:
\begin{eqnarray*}
P\{Y^*\geq a\}  &=  &  1 - P\{ T=\infty\}  \\
& = &   1 -   \exp\left(  - \int_0^\infty  \kappa(y(t))   dt    \right)  \\
& = &   1 -   \exp\left(  - \int_0^\infty  \kappa(y)   \left\{  \mu + \frac  {\sigma^2}{y+a}  \right\}^{-1}  dy  \right)  \\
& = & 1 - \exp \left( - \int_0^{\infty}  \left\{      \frac{1}{y+a} - \frac{\mu}{y\mu +a\mu+\sigma^2}   \right\} dy  \right) \\
& =  &  \frac{1}{1+ \frac{a\mu}{\sigma^2}}
\end{eqnarray*}
Thus, the process $Y$ satisfies the bound of Proposition \ref{prop.drift_down} with equality for $\gamma=\frac{\mu}{\sigma^2}.$\\

\paragraph{Example 2: Proof of Proposition \ref{prop.uniform}.}   Take $ h(y)=b+y$ for some $b > 0.$ 
Equation \eqref{eq.Idef} yields that for $b\geq 0$,
$$P\{Y^* \geq b+c\}  = \exp(-I(c)) -\exp( -I(\infty))$$
where
\begin{eqnarray*}
I(c) & = &    \int _0^c    \frac{\sigma^2}{(b+2y)^2}  \left\{ \mu + \frac{\sigma^2}{b + 2y} \right\}^{-1} dy  \\
& = &  \int _0^c  \left\{  \frac{1}{b+2y} - \frac{\mu}{\mu(b+2y) + \sigma^2} \right\} dy  \\
& = & \frac{1}{2}\ln\left(  \frac{b+2c}{\mu(b+2c)+\sigma^2}\right) - \frac{1}{2}\ln\left( \frac{b}{\mu b + \sigma^2 }\right)
\end{eqnarray*}
Using this and \eqref{eq.Idef}, and setting $c = a - b,$  yields
\begin{equation}
P\{ Y^* \geq a \} = \left\{  \begin{array}{cl}  \left(  \frac{\mu b}{\mu b + \sigma^2}\right)^{\frac{1}{2}}
\left\{  \left( 1 + \frac{\sigma^2}{\mu(2a-b)}\right)^{\frac{1}{2}} -1 \right\}  &  a\geq b    \\
1 - \left(  \frac{\mu b }{\mu b + \sigma^2 }\right)^\frac{1}{2}  &  0 < a  \leq b
 \end{array} \right.
\end{equation}
Let $b=\frac{16 \sigma^2}{9\mu}.$   Then $P\{Y^*\geq a\} =  \frac{1}{5} \geq   \frac{1}{5(1+\frac{\mu a}{\sigma^2} ) }$ for $0 \leq  a \leq b.$
By checking derivatives, it is easy to verify that  $(1+ \frac{\alpha}{2})^\frac{1}{2} -1 \geq \frac{\alpha}{4(1+\alpha)}$ for any $\alpha > 0$.
Therefore, for this choice of $b$, and $a \geq b$,
\begin{eqnarray*}
P\{ Y^* \geq a \} & \geq   & \frac{4}{5}
\left\{  \left( 1 + \frac{\sigma^2}{2\mu a }\right)^{\frac{1}{2}} -1 \right\}  \\
& \geq & \frac{1}{5(1+\frac{\mu a}{\sigma^2} ) }.
\end{eqnarray*}
This bound for the process $Y$ proves Proposition \ref{prop.uniform}.

\section{Reformulation and proof of Proposition \ref{prop.drift_down}}  \label{sec.dynamic_programming}

Proposition  \ref{prop.drift_down} concerns the probability that a process with downward drift, starting from $0$,
reaches level $a.$    It is more convenient for the proof, to consider the equivalent problem, of the probability a process
with upward drift, starting from $a$, reaches zero.  The correspondence between the two formulations is obtained
by setting $X=a-Y,$ so the proof of Proposition \ref{prop.mgbndZ}  below also establishes Proposition \ref{prop.drift_down}.

\begin{cond}   \label{cond.drift_up}
$X$ is a semimartingale and  $\gamma \geq 0,$ such that  $X - \gamma [X,X]$  is a submartingale.
\end{cond}

\begin{proposition} \label{prop.mgbndZ}  (Equivalent to Proposition \ref{prop.drift_down}.)
Suppose $X$ and $\gamma$ satisfy Condition \ref{cond.drift_up} and $X_0=a$ for some $a\geq 0.$
Let $T=\inf\{ t : X_t \leq 0\}$  (so $T=\infty$ if $X_t > 0$ for all $T$).  \\
(a) The following holds:
\begin{equation}   \label{eq.mg_bound_up}
P\{ T < \infty \} \leq  \frac{1}{1+\gamma a}.
\end{equation}
(b) Equality holds in \eqref{eq.mg_bound_up} if and only if
$(X_{t\wedge T}: t\geq 0)$ has no continuous martingale component,   $X$ is sample-continuous over
$[0,T)$ with probability one, $P\{X_T=0| T < \infty\}=1$, and $\left( (X -  \gamma [X,X])_{t\wedge T}: t \geq 0\right)$ is a martingale.
\end{proposition}

\begin{proof}
{\em (Proof of (a))}
Suppose $X$ and $\gamma$ satisfy Condition \ref{cond.drift_up} and $X_0=a$ for some $a\geq 0.$
Let $\tilde{X}_t= \max\{ X_{t\wedge T}  , 0 \}.$    Let $D=-X_T$ on the event $T < \infty$, and let $D=0$ otherwise.
Note that $X_t=\tilde{X}_t$ and  $[\tilde{X},\tilde{X}]_t  =  [X,X]_t $ for $t\in [0,T),$ 
while, if $\{T< \infty\}$, 
 $0=\tilde{X}_T \geq X_T = -D $   and    $\triangle[\tilde{X},\tilde{X}]_T =   (X_{T-})^2 \leq  (D + X_{T-})^2  =\triangle[X,X]_T.$
Thus,
\begin{equation}  \label{eq.compare}
(\tilde{X}-\gamma [\tilde{X},\tilde{X}])_t - (X - \gamma [X,X])_{t\wedge T} = \left\{ \begin{array}{cl}  0 & 0\leq t < T \\
D + \gamma(  2DX_{T-}+D^2  )  \geq 0 &  t \geq T
\end{array} \right.
\end{equation}
Since $X-\gamma [X,X]$ is a submartingale, so is $(X-\gamma [X,X])_{t\wedge T},$ and therefore so is $\tilde{X}-\gamma [\tilde{X},\tilde{X}].$   
Therefore, $\tilde{X}$ and $\gamma$ also satisfy the conditions of the proposition, and $T=\inf\{t\geq 0 : \tilde{X} \leq 0\}.$     Therefore,
to prove (a), it suffices to prove (a) with $X$ replaced by $\tilde{X}.$  Equivalently, in proving (a), we can, and do, make
the assumption, without loss of generality, that $X_t \equiv 0$ for $t\geq T.$

Let $p(x)=\frac{1}{1+\gamma x}$ for $x \geq 0.$
 Let $0\leq s < t.$   By the Dol\'{e}an-Dade Meyer change of variables formula for semimartingales,
\begin{equation}  \label{eq.COV}
 \begin{split}
p(X_t)  & =  p(X_s) + \int_s^t p'(X_{u-})dX_u  \\
&  +  \sum_{s<u\leq t} \left( p(X_u)-p(X_{u-})-p'(X_{u-})\triangle X_u\right)   +
 \frac{1}{2}\int_s^t p''(X_u)d[X,X]_u^c  .
\end{split}
\end{equation}
Observe that
\begin{eqnarray}  \label{eq.pprime}
p'(X_{u-})= - \gamma  p^2(X_{u-})
\end{eqnarray}
\begin{eqnarray}   
p(X_u)-p(X_{u-})-p'(X_{u-})\triangle X_u & = & (\triangle X_u)^2\gamma^2p^2(X_{u-})p(X_u)  \nonumber \\
& \leq & (\triangle X_u)^2  \gamma^2 p^2(X_{u-})    \label{eq.ptriangle}
\end{eqnarray}
and
\begin{eqnarray}
p''(X_{u-})& = &2\gamma^2p^3(X_{u-})  \nonumber  \\
& \leq &  2\gamma^2 p^2 (X_{u-}) .  \label{eq.pdoubleprime}
\end{eqnarray}

Combining \eqref{eq.COV} - \eqref{eq.pdoubleprime} and the fact $[X,X]_u = [X,X]^c_u  + \sum_{v \leq u} (\triangle X_v)^2$ yields
\begin{equation}
p(X_t) \leq p(X_s) - \gamma  \int_s^t p(X_{u-})^2 dG_u
\end{equation}
where   $G  = X -  \gamma  [X,X] .$
By assumption, $G$ is a submartingale, so
\begin{equation}   \label{eq.Ginequal}
E\left[   \int_s^t p(X_{u-})^2 dG_u  \bigg|  {\cal F}_s \right] \geq 0
\end{equation}
Therefore, $p(X)$ is a supermartingale, so that $E[p(X_t)] \leq p(X_0) = p(a). $  For $t\geq 0,$
$\{T\leq t\} = \{ p(X_t)=1\}.$ 
Therefore, $P \{ T\leq t\}  \leq E[p(X_t)] \leq  p(a) $ for all $t$.  
Thus, $P\{T< \infty\} = \lim_{t\rightarrow \infty} P\{T\leq t\} \leq p(a)$, completing the proof of part (a).

{\em (Proof of (b))}   To begin, assume that $X_t \equiv 0$ for $t \geq T,$ with probability
one.   Note  that there are three places in the proof of  part (a) where the bounds
might not hold with equality.   These are associated with inequalities  \eqref{eq.ptriangle}, 
\eqref{eq.pdoubleprime}, and \eqref{eq.Ginequal}.
First,  if $\triangle X_s \neq 0$, then  \eqref{eq.ptriangle} holds with equality if and only if  $X_s \neq 0$.
Therefore, in order for \eqref{eq.mg_bound_up} to be tight, with probability one, any jumps of $X$
happening in $[0,T]$ must happen at time $T.$    Second, the inequality  \eqref{eq.pdoubleprime}
for $p''(X_{u-})$ is strict for any $u < T$.   This inequality is applied to the
last term in \eqref{eq.COV}.   Therefore, the resulting upper bound on the last term in \eqref{eq.COV} holds with
equality if and only if $[X,X]^c \equiv 0$ over the interval $(s,t].$    Finally, \eqref{eq.Ginequal} holds
with equality if and only if $G$ is a martingale (not just a supermartingale) over $(s,t].$   
These observations prove (b) under the additional assumption that $X_t \equiv 0$ for $t \geq T.$

To prove (b) in general, suppose $X$ satisfies the conditions of the proposition, and let $\tilde{X}$
be defined in part (a).   The conditions in part (b) involve $X$ only up to time $T.$   The process $X$ and 
$\tilde{X}$ are equal for $t<T$, so the only difference between them over $[0,T]$
could possibly be at time $T$.      If   \eqref{eq.mg_bound_up} holds with equality, then by the special case
of part (b) already proved, $\tilde{X}$ must satisfy the conditions in (b).   In particular,
$(X-\gamma [X,X])_{t\wedge T}$ is a martingale.   In view of \eqref{eq.compare},
$(X-\gamma [X,X])_{t\wedge T}$ must also be a martingale and
$P\{ D=0\}=1$, or equivalently, $P\{X_T=0 | T<\infty\}=1.$   Therefore, $X$ also satisfies
the conditions in (b), and the proof of the proposition is complete.
\end{proof}

The above proof is inspired by dynamic programming.   The problem of finding $X$ to maximize $P\{ T < \infty \}$ subject
to the given constraints can be viewed as a stochastic control problem.   Intuitively speaking, given $X_s=x$ for
some time $s$ and $x>0$, the conditional distribution of the increment $X_{s+\eta}-X_s$ must be selected subject to first
and second  moment constraints.   The function $p$ plays the role of  the value function in dynamic programming.

\section{Discrete time processes}   \label{sec.discrete_time}

Suppose $(\Omega, {\cal F}, P)$ is a complete probability space with
a filtration of $\sigma$-algebras $({\cal F}_k : k\in \IZ_+)$.

\begin{cond}  \label{cond.drift_down_discrete}
$S=(S_k: k\in \IZ_+)$  is an adapted random process with $S_0=0$ and, with $U_j=S_{j}-S_{j-1}$  for $j\geq 1$,
$(S_k  +  \gamma \sum_{1 \leq j \leq k} U_j^2: k\geq 0)$ is a supermartingale.
\end{cond}

Let $S^*=\sup\{ S_k: k\in \IZ_+\}.$

\begin{proposition} \label{prop.drift_down_discrete}
Suppose $S$ and $\gamma \geq 0$ satisfy Condition \ref{cond.drift_down_discrete} and $a\geq 0.$  \\
(a)  The following holds:
\begin{equation}  \label{eq.mg_bound_down_discrete}
P\{  S^*  \geq a \}   \leq \frac{1}{1+\gamma a}
\end{equation}
(b)  For any $\gamma\geq 0, a\geq 0$ and $\epsilon > 0$, there is a process $S$ satisfying
Condition \ref{cond.drift_down_discrete} such that  $P\{  S^*  \geq a \}   \geq \frac{1}{1+\gamma a} - \epsilon.$
\end{proposition}

\begin{proof}
{\em (Proof of (a))}   The filtration $({\cal F}_k : k\in \IZ_+)$ can be extended to a filtration $({\cal F}_t : t\in \IR_+)$ by letting ${\cal F}_t={\cal F}_{\lfloor t \rfloor}$
for $t\in \IR_+$, and the process $S$ can be extended to a piecewise constant process $(Y_t : t\in \IR_+)$ by letting $Y_t=S_{\lfloor t \rfloor}$
for $t\in \IR_+.$   Then $S^*=Y^*$ and $Y$ satisfies Condition \ref{cond.drift_down}.   Thus, by
Proposition  \ref{prop.drift_down},   $P\{S^*\leq a\}=P\{Y^*\leq a\} \leq  \frac{1}{1+\gamma a}.$  This establishes (a).

{\em (Proof of (b))}   If $\gamma = 0$, the process $S$ can be a mean zero random walk with finite variance jumps, and then
$S^*=\infty$ with probability one, so that equality holds in  \eqref{eq.mg_bound_down_discrete} for such choices of $S.$  
So assume for the remainder of the proof that $\gamma > 0.$
Given $\tilde{\mu} > 0,$   let $\tilde{\sigma}^2=\frac{\tilde{\mu}}{\gamma}-\tilde{\mu}^2.$   Assume that $\tilde{\mu}$ is so small that
$\tilde{\sigma}^2 > 0.$     Let $\tilde{a}=a+ \tilde{\mu}.$   Let $Y$ be a continuous time martingale as constructed in Example 1,
for the parameters $\tilde{\mu}, \tilde{\sigma}^2,$ and $\tilde{a}.$     Thus, $Y$ follows a deterministic trajectory up to some random
time $T$, which could be infinite.   If $T < \infty$, then $Y$ jumps up to $\tilde{a}$ at time $T$.    Furthermore, modify $Y$ on the event
$\{T < \infty\}$ by letting $Y_t =\tilde{a}-\tilde{\mu}(t-T)$ for $t\geq T.$   Then $Y_t-\tilde{\mu}t$ and $[Y,Y] -\tilde{\sigma}^2(t\wedge T)$ are martingales,
and $P\{Y^*\geq \tilde{a} \}=\frac{1}{1+\tilde{\gamma}\tilde{a}},$  
 where $\tilde{\gamma}=\frac{\tilde{\mu}}{\tilde{\sigma}^2}.$   Let $M$ denote the martingale defined by $M_t=Y_t - \tilde{\mu} t.$   Since $M$ differs
 from $Y$ by a continuous function with finite variation, $[M,M] \equiv [Y,Y]$ with probability one.  Since $M$ is a martingale with finite second moments,
 $M^2-[M.M]$ is also a martingale.
 
 Let $S=(S_k : k \in \IZ_+)$ be obtained by sampling $Y$ at nonnegative integer times, namely,  $S_k=Y_k$   for $k\in \IZ_+.$
We first check that $S$ satisfies Condition \ref{cond.drift_down_discrete}. 
For $k\geq 1,$  set $U_k=S_k - S_{k-1}=Y_{k}-Y_{k-1}.$  Then
$E[U_k | {\cal F}_{k-1}] = E[Y_{k}-Y_{k-1} | {\cal F}_{k-1}]  = \tilde{\mu},$
and
\begin{eqnarray*}
E[U_k^2  | {\cal F}_{k-1}]  & = & E[(\tilde{\mu}+M_k-M_{k-1})^2  | {\cal F}_{k-1}]  \\
& = &  \tilde{\mu}^2 +    E[(M_k-M_{k-1})^2  | {\cal F}_{k-1}]   \\
& = & \tilde{\mu}^2 +   E[ [M,M]_{k} - [M,M]_{k-1}  | {\cal F}_{k-1}]  \\
& = &  \tilde{\mu}^2 +   E[ [X,X]_{k} - [X,X]_{k-1}  | {\cal F}_{k-1}]  \leq  \tilde{\mu}^2  + \tilde{\sigma}^2
\end{eqnarray*}
Therefore, by the choice of $\tilde{\sigma}^2,$
$$
E[ U_k - \gamma U_k^2  | {\cal F}_{k-1}] \geq  \tilde{\mu} -\gamma ( \tilde{\mu}^2  + \tilde{\sigma}^2 ) = 0
$$
Thus, $S$ satisfies Condition  \ref{cond.drift_down_discrete} as required.
On the event $\{T < \infty\}$,  $Y_T=\tilde{a}=a+\mu$ and $Y$ decreases with constant slope $-\tilde{\mu}$ after time $T$, 
implying that  $S_{\lceil T \rceil} = Y_{\lceil T \rceil} \geq a$, so that $S^* \geq a$.
Therefore
\begin{equation}   \label{eq.SY}
P[S^* \geq a] \geq P[Y^* \geq \tilde{a} ] =\frac{1}{1+\tilde{\gamma}\tilde{a}} = \frac{1}{1+\left[\frac{\gamma}{1+\tilde{\mu}\gamma}\right][\tilde{\mu}+a ]}
\end{equation}
For $\tilde{\mu}$ sufficiently small, \eqref{eq.SY} implies the inequality of part (b), and the proof is complete.
\end{proof}

\section{Discussion}   \label{sec.discussion}

There is a large literature on bounds on processes implied by drift conditions and
concentration inequalities.  Typically the constraints placed on the sizes of the increments of the processes
are stronger than those imposed here.    For example, in the discrete-time case,  Condition \ref{cond.drift_down_discrete}
constrains the second moment of the increments in terms of the drift, but stronger constraints
such as bounded increments, or uniformly exponentially dominated increments \cite{hajek82drift}, yield stronger bounds.
The motivation of this paper is to explore the implication of a relatively mild condition on the
variation of a process with negative drift.  A second motivation of this paper is that the bound represents
the solution of an interesting optimal stochastic control problem that may arise in some applications such as finance.

The proof of Proposition \ref{prop.drift_down} might leave the reader wanting some additional insight
into why the big jump processes are the ones meeting the bound with equality.    As noted in Section \ref{sec.dynamic_programming},
it is useful to view the bound $p(a)=\frac{1}{1+  \gamma a}$ as a value function in the sense of dynamic programming.
Given its explicit form, it is trivial to verify that it is decreasing, convex, and the third derivative is negative.  However,
one might expect these properties even without knowing the function explicitly.   The fact it is decreasing means that the chances
of $X$ ever reaching zero decrease with the initial state $a$.   Due to Jensen's inequality, the fact the function is convex means that the chances
of ever reaching zero from an initial state $a$ would increase if, before imposing the constraints on $X$,  the process could
follow a martingale trajectory.      Finally, the fact that the third derivative is negative means that the second derivative is decreasing. 
Therefore, it is better to use the variations of $X$ at smaller values. In this connection, it is useful to remember that
martingales, even discontinuous ones, can be obtained from time changes of brownian motion.  That explains why the process,
for initial state $a$, has only two possible values at future times, and why there are no upward jumps in the process.   

We close by speculating about two possible extensions.
Proposition \ref{prop.drift_down} shows that, for continuous-time processes, even though
Condition  \ref{cond.timed_drift_down} is more restrictive than Condition  \ref{cond.drift_down},
the same maximum value of $P\{ Y^* \geq a\}$ can be achieved.   However, the situation is different
in discrete-time.  The discrete-time version of Condition  \ref{cond.timed_drift_down} would constrain the
process $Y$ to take downward jumps that are not vanishingly small, and the maximum possible value of
$P\{ Y^* \geq a\}$ would be strictly smaller than the bound in Proposition  \ref{prop.drift_down_discrete}.
The intuition in the first paragraph above would indicate that
in such a case, the optimal process would still be one following a deterministic path, with
the possible exception of a single big jump, reaching the target value in one step.

Proposition \ref{prop.uniform} shows that under Condition  \ref{cond.drift_down} it is possible
that $E[Y^*] = +\infty.$ However, it might be interesting to consider maximizing the expected value of sublinear
functions of $Y$, such as $E[(Y^*)^\alpha]$ for a fixed $\alpha$ with $0< \alpha < 1.$   The process
of Proposition \ref{prop.uniform} achieves at least one-fifth of the maximum possible value.
Perhaps the big jump construction of Section \ref{sec.construction} still yields the extremal processes,
for a suitable choice of the function $h.$

\section{Related inequalities}  \label{sec.related}

\subsection{Kingman's bound}

If $Y$  has stationary, independent increments (SII) and $Y$ and $\gamma$
satisfy Condition \ref{cond.timed_drift_down}, the moment upper bound of Kingman \cite{Kingman62} for
the waiting time in a $GI/G/1$ queue implies\footnote{Kingman's moment bound is for discrete time SII
processes, but it is easily extended to
continuous time by sampling the continuous time processes at times of the form $2^{-n},$ and letting $n\rightarrow\infty.$}
\begin{equation}   \label{eq.Kingman}
E[Y^*]  \leq \frac{1}{2\gamma},   ~~~~~\mbox{(under SII)}
\end{equation}
which together with Markov's inequality implies that
$$
P\{ Y^*  \geq a \}  \leq  \frac{1}{2a\gamma}. ~~~~~\mbox{(under SII)}
$$
If $Y$ is skip-free positive, a simple argument shows that $Y^*$ has an exponential distribution, so that
\begin{equation}
P\{ Y^* \geq a \}  \leq \exp(-2\gamma a)  ~~~~~\mbox{(under SII, skip-free positive)}  \label{eq.SIIsf}
\end{equation}
If $Y$ is SII and  both skip-free positive and skip-free negative, it is, of course, continuous.
If in addition both $(Y_t +\mu t : t \geq 0) $ and $([Y,Y] - \sigma^2 t : t\geq 0)$ are martingales,
$Y$ can be written as $Y=\sigma B_t - \mu t,$  where $B$ is a standard Brownian motion,  in which case
equality holds in both \eqref{eq.Kingman} and  \eqref{eq.SIIsf}.

\subsection{Doob's $L^p$ inequalities}

Doob's well known $L^p$  inequality for a nonnegative submartingale $X=(X_t: 0 \leq t \leq T)$,
holding for $p>1$, is:
\begin{equation}  \label{eq.suppbnd}
||X^*||_p   \leq  \frac{p}{p-1}  ||X_T||_p.
\end{equation}
Dubins and Gilat \cite{DubinsGilat78} gave a construction of a martingale showing that the constant
in \eqref{eq.suppbnd} is the best possible.   That martingale can be expressed as a big jump process
as follows.   Let $h$ be a positive,  nondecreasing function on
the interval $[0,1],$  let $U$ be uniformly distributed on the interval $[0,1],$    let $0<c<1,$
and let $X=(X_t: 0\leq t \leq 1)$ denote the process:
$$
X_t = \left\{  \begin{array}{cl}  h(t)  &  t < U   \\     ch(U)  & t\geq U. \end{array} \right.
$$
In words, $X$ follows $h$ until time $U$, at which time it jumps downward to a fraction $c$ of itself, and sticks.
We now determine a choice of $h$ for
which $X$ is a martingale.   The failure rate function of $U$ (or jump intensity, given the jump
hasn't yet happened) at time $t$ is $\frac{1}{1-t}.$  Therefore,  the drift of $X$ at time  $t$ is
$h'(t)-  \frac{(1-c)h(t)}{1-t}.$   Setting the drift to zero
determines $h$, up to a constant factor, to be   $h(t)=\frac{1}{(1-t)^{1-c}}.$    
Let $T=1.$   Note that $X^*=h(U)$ and $X_T=cX^*.$    Given $p  >  1$,
$X_T$ is in $L^p$ if $(1-c)p>1.$   We thus have:
$$
||X^*||_p  = \frac{1}{c} ||X_T||_p <\infty  ~~\mbox{if}~~\frac{1}{c} >  \frac{p}{p-1},
$$
which shows that the constant in \eqref{eq.suppbnd} is the best possible.   

Cox \cite{Cox84} described the dynamic programming approach for Doob-like inequalities, and
showed that Doob's $L^2$ martingale inequality can't be satisfied with equality. 

The analysis of Dubins and Gilat \cite{DubinsGilat78} is related to work of Blackwell and Dubins \cite{BlackwellDubins63},
which makes a connection to the Hardy-Littlewood maximal function \cite{HardyLittlewood30},
$h$, of a nondecreasing integrable function $g$ on $[0,1]$, defined as follows:
$$
h(t) = \frac{1}{1-t}\int_t^1 g(u) du.
$$
In fact,  $h$ is the unique function such that a process that follows $h$ up until time $U$, and then jumps and
sticks to value $g(U),$  is a martingale.
Blackwell and Dubins \cite{BlackwellDubins63} used the inequality
$$\lambda \leq  \frac{   \int_{\{X^*\geq \lambda\}} X_T  dP   }{P{\{X^*\geq \lambda\}} }
$$
to show that for a specified distribution of the end variable $X_T,$   the  Hardy-Littlewood maximal
function \cite{HardyLittlewood30}  provides the stochastically largest possible distribution for 
the supremum of a martingale.   The idea is the following.   Given the distribution of $X_T$,  one can
ask how large $P\{E\}$ can be for an event $E$ such that
$$\lambda \leq  \frac{   \int_E    X_T  dP   }    {P\{E \}} .
$$
Without loss of generality assume that the underlying probability space is $[0,1]$ and that
$X_T=g$, where $g$ is a nondecreasing function on $[0,1].$   
Clearly $E$ should be an event of the form $[\gamma, 1],$ and 
$$\max _E P(E)=1-\min\left\{t\geq 0  :  \frac{ \int_t^1 g(t) dt }{1-t}  \geq  \lambda\right\}.$$
So if $h(t)=\frac{ \int_t^1 X(t) dt}{1-t}$, $P(E)  \leq P\{ h \geq \lambda  \}$.
Thus, $P\{X^* \geq \lambda \} \leq P\{ h \geq \gamma \}.$   That is, $X^*$ is stochastically
dominated by $h.$

In contrast, the inequalities given in Proposition \ref{prop.drift_down} can be achieved with equality.  Also,
 there is no maximal probability distribution for $Y^*$ subject to Condition \ref{cond.drift_down},
 because equality cannot hold in \eqref{eq.mg_bound_down} for two distinct positive values of $a.$

\end{document}